\documentclass[12pt,a4paper]{amsart}
\usepackage{amsfonts,amscd,amsmath,amssymb,amsthm,mathrsfs}
\usepackage{array,latexsym,float,enumerate}
\usepackage{graphicx,epsfig,xcolor,ifpdf}
\usepackage[top=3cm, bottom=3cm, left=2.5cm, right=2.5cm,twoside=false]{geometry}
\usepackage[all]{xy}
\numberwithin{equation}{section}

\theoremstyle{plain}
\newtheorem{theorem}{Theorem}[]

\newtheorem{thmx}{Theorem}

\newtheorem{lemma}[theorem]{Lemma}
\newtheorem{proposition}[theorem]{Proposition}
\newtheorem{corollary}[theorem]{Corollary}
\theoremstyle{definition}
\newtheorem{definition}[theorem]{Definition}
\newtheorem{example}[theorem]{Example}
\newtheorem{noname}[theorem]{}
\newtheorem{subnoname}{}[theorem]
\newtheorem{remark}[theorem]{Remark}
\newtheorem{construction}[theorem]{Construction}
\newtheorem{notation}[theorem]{Notation}
\theoremstyle{remark}
\newtheorem*{smallremark}{Remark}

\newtheorem{case}{Case} \makeatletter \@addtoreset{case}{theorem}\makeatother
\newtheorem{claim}{Claim} 

\newcommand{\bthm}{\begin{theorem}}
\newcommand{\bprop}{\begin{proposition}}
\newcommand{\blem}{\begin{lemma}}
\newcommand{\bcor}{\begin{corollary}}
\newcommand{\brem}{\begin{remark}}
\newcommand{\bdfn}{\begin{definition}}
\newcommand{\bitem}{\begin{itemize}}
\newcommand{\benum}{\begin{enumerate}}
\newcommand{\bex}{\begin{example}}
\newcommand{\bno}{\begin{noname}}
\newcommand{\bsno}{\begin{subnoname}}
\newcommand{\bsrem}{\begin{smallremark}}
\newcommand{\bnot}{\begin{notation}}
\newcommand{\bcon}{\begin{construction}}
\newcommand{\bca}{\begin{case}}
\newcommand{\bcl}{\begin{claim}}
\newcommand{\beq}{\begin{equation}}
\newcommand{\bpf}{\begin{proof}}

\newcommand{\epf}{\end{proof}}
\newcommand{\eeq}{\end{equation}}
\newcommand{\ecl}{\end{claim}}
\newcommand{\eca}{\end{case}}
\newcommand{\econ}{\end{construction}}
\newcommand{\enot}{\end{notation}}
\newcommand{\esrem}{\end{smallremark}}
\newcommand{\eno}{\end{noname}}
\newcommand{\esno}{\end{subnoname}}
\newcommand{\eex}{\end{example}}
\newcommand{\eitem}{\end{itemize}}
\newcommand{\eenum}{\end{enumerate}}
\newcommand{\ethm}{\end{theorem}}
\newcommand{\eprop}{\end{proposition}}
\newcommand{\elem}{\end{lemma}}
\newcommand{\ecor}{\end{corollary}}
\newcommand{\erem}{\end{remark}}
\newcommand{\edfn}{\end{definition}}
\newcommand{\ble}{\begin{lemma}}
\newcommand{\ele}{\end{lemma}}

\newcommand{\ov}{\overline}

\newcommand{\wt}{\widetilde}

\def\8{\infty}
\def\.{\cdot}
\def\PP{\mathbb{P}}
\def\F{\mathbb{F}}
\def\C{\mathbb{C}}
\def\Z{\mathbb{Z}}

\def\Q{\mathbb{Q}}

\def\E{\widehat{E}}

\def\xra{\xrightarrow}
\def\raa{\xrightarrow{\ \ }}

\def\:{\colon}
\def\map{\dashrightarrow}

\def\ssk{\smallskip}

\def\Sing{\operatorname{Sing}}

\def\dim{\operatorname{dim}}

\def\Sing{\operatorname{Sing}}

\def\dim{\operatorname{dim}}

\def\id{\operatorname{id}}

\def\qt{\operatorname{qt}}

\newcommand{\ti}{\tilde}

\ifpdf \usepackage[linkbordercolor={0 0 1}]{hyperref} \else \usepackage[hypertex,linkbordercolor={0 0 1}]{hyperref} \fi

\begin{document}

\title[Theorems of Lin-Zaidenberg and Abhyankar-Moh-Suzuki]{A new proof of the theorems of Lin-Zaidenberg \\and Abhyankar-Moh-Suzuki}

\author[Karol Palka]{Karol Palka}
\address{Karol Palka: Institute of Mathematics, Polish Academy of Sciences, ul. \'{S}niadeckich 8, 00-656 Warsaw, Poland}\email{palka@impan.pl}

\begin{abstract} Using the theory of minimal models of quasi-projective surfaces we give a new proof of the theorem of Lin-Zaidenberg which says that every topologically contractible algebraic curve in the complex affine plane has equation $X^n=Y^m$ in some algebraic coordinates on the plane. This gives also a proof of the theorem of Abhyankar-Moh-Suzuki concerning embeddings of the complex line into the plane. Independently, we show how to deduce the latter theorem from basic properties of $\Q$-acyclic surfaces.\end{abstract}

\thanks{The author was supported by the Polish Ministry of Science and Higher Education, the Iuventus Plus grant, No. 0382/IP3/2013/72}


\maketitle

The following result is a homogeneous formulation of the theorems proved by Lin-Zaidenberg \cite{LinZaidenberg} and Abhyankar-Moh \cite{AbhMoh_the_line_thm} and Suzuki \cite{Suzuki_AMSthm}. Curves are assumed to be irreducible.

\begin{thmx} If a complex algebraic curve in $\C^2$ is topologically contractible then in some algebraic coordinates $\{x,y\}$ on $\C^2$ its equation is $x^n=y^m$ for some coprime $n>m>0$.
\end{thmx}

The part proved by Lin-Zaidenberg concerns singular curves ($m\geq 2$). The part proved by Abhyankar-Moh and Suzuki concerns smooth curves ($m=1$) and is usually stated in the following form.

\begin{thmx} If a complex algebraic curve in $\C^2$ is isomorphic to $\C^1$ then in some algebraic coordinates $\{x,y\}$ on $\C^2$ its equation is $x=0$.
\end{thmx}

The Theorem B has now several published proofs using variety of methods, from algebraic to topological. The easiest we know is by Gurjar \cite{Gurjar_AMS}. As for the singular case of Theorem A, the original proof relies on Teichm\"uller theory. A topological proof based on properties of knots was given in \cite{NeumanRudolph_unfoldings, NeumanRudolph_correction}. Proofs of both theorems using algebraic geometry can be found in \cite{GurjarMiyanishi_AMS_and_LZ_thms} and \cite{Koras-ab_moh}. The latter two use the tools of the theory of open algebraic surfaces including the logarithmic Bogomolov-Miyaoka-Yau inequality established for surfaces of log general type by Kobayashi \cite{Kobayashi-uniformization} and Kobayashi-Nakamura-Sakai \cite{KobNakSak_logBMY_inequality}. Our proof of Theorem A also uses the theory of minimal models for log surfaces. We believe it is quite short and geometric. Both theorems are deduced from the following result.

\begin{theorem}\label{thm:main} If $A\subseteq \C^2$ is a topologically contractible curve then there exists a minimal smooth completion $(X,D)$ of $\C^2\setminus A$, such that the proper transform of $A$ is a fiber of a $\PP^1$-fibration of $X$, whose restriction to $\C^2\setminus A$ has irreducible fibers.
\end{theorem}

The basic new ingredient in the proof is to shift the focus from the surface $\C^2\setminus A$, where $A$ is the contractible curve, to the surface $X=(\C^2\setminus A)\cup C$, where $C$ is (some naturally defined open subset of) the last $(-1)$-curve created by the minimal log resolution of the singularity at infinity (see section \ref{sec:preliminaries}). While the boundary of $X$ is not any more connected, the important property is that in general the Euler characteristic of $X$ is negative. A similar idea was used in \cite{PaKo-general_type} and will be used in forthcoming papers (coauthored with M. Koras and P. Russell) finishing the classification of closed $\C^*$-embeddings into $\C^2$. Another new ingredient is that we rely on a more general version of the log BMY inequality which works for surfaces of non-negative logarithmic Kodaira dimension. We tried to make the article self-contained. In section \ref{sec:AMS} we give an independent, direct proof of Theorem B using some basic properties of $\Q$-acyclic surfaces.

The author would like to thank Peter Russell and the referee for a careful reading of the preliminary version of the article.

\tableofcontents

\section{Preliminaries and notation}\label{sec:preliminaries}

\ssk We work in the category of complex algebraic varieties. The results of this section are well known, we give short proofs for completeness. Let $D=\sum_{i=1}^n D_i$ be a reduced effective divisor on a smooth projective surface, which has smooth components $D_i$ and only normal crossings (i.e.\ $D$ is an \emph{snc-divisor}). The number of (irreducible) components of $D$ is denoted by $\#D$. A component $C$ of $D$ is \emph{branching} if it meets more then two components of $D-C$. A $(k)$-curve is a curve isomorphic to $\PP^1$ which has self-intersection $k$. We say that $D$ is \emph{snc-minimal} if after a contraction of any $(-1)$-curve contained in $D$ the image of $D$ is not an snc-divisor, or equivalently, if every $(-1)$-curve of $D$ is branching. We define the discriminant of $D$ by $d(D)=\det(-[D_i\cdot D_j]_{i,j\leq n})$. We put $d(0)=1$.

A \emph{smooth pair} $(\ov X,D)$ consists of a smooth projective surface $\ov X$ and a reduced snc-divisor $D$. If $X$ is a smooth quasi-projective surface then a \emph{smooth completion} of $X$ is any smooth pair $(\ov X,D)$ with a fixed identification $\ov X\setminus D\cong X$.  By $\rho(\ov X)$ we denote the Picard rank of $\ov X$, which in case of rational surfaces is the same as $\dim H_2(\ov X,\Q)$. We check easily that the discriminant of $D$ and its total reduced transform under a blowing up are the same. Thus, for smooth completions of $X$ the discriminant $d(D)$ depends only on $X$. In particular, if $X=\C^2$ then $d(D)=-1$. The following formula is a consequence of elementary properties of determinants.

\setcounter{theorem}{0}
\blem Let $T_1$ and $T_2$ be reduced snc-divisors for which $T_1\cdot T_2=1$ and let $C_1$, $C_2$ be the unique components of $T_1$ and $T_2$ which meet. Then \begin{equation}d(T_1+T_2)=d(T_1)d(T_2)-d(T_1-C_1)d(T_2-C_2).\label{eq:splitting_d}\end{equation}
\elem

A $\PP^1$- (a $\C^1$-) fibration is a surjective morphism whose general fibers are isomorphic to $\PP^1$ (respectively to $\C^1$). If $(\ov X,D)$ is a smooth completion of $X$ and $p$ is some fixed $\PP^1$-fibration of $\ov X$ we put $\Sigma_X=\sum_{F\not\subset D}(\sigma(F)-1)$, where the sum is taken over all fibers of $p$ not contained in $D$ and $\sigma(F)$ is the number of components of $F$ not contained in $D$. Clearly, $\Sigma_X\geq 0$ and the equality holds if and only if the restriction $p_{|X}$ has irreducible fibers. Let $\nu$ and $h$ be respectively the number of fibers contained in $D$ and the number of horizontal components of $D$ (i.e.\ these whose push-forward by $p$ does not vanish). The following lemma is due to Fujita \cite[4.16]{Fujita-noncomplete_surfaces}.

\blem\label{lem:Sigma} Let $(\ov X,D)$ be a smooth completion of a smooth surface $X$. With the above notation for every $\PP^1$-fibration of $\ov X$ we have  \begin{equation}h+\nu+\rho(\ov X)=\Sigma_X+\#D+2\label{eq:Sigma}.  \end{equation}
\elem

\begin{proof} Having a $\PP^1$-fibration, $\ov X$ dominates some $\PP^1$-bundle over a projective curve.  The latter has $\rho=2$, so we have $\rho(\ov X)-2=\sum_F(\#F-1)=\sum_{F}\#F\cap D+\sum_{F}(\sigma(F)-1)=(\#D-h)+\Sigma_X-\nu.$
\end{proof}

It is well known that every singular fiber of a $\PP^1$-fibration of a smooth projective surface can be inductively reconstructed from a $0$-curve by blowing up. In particular, we deduce by induction the following lemma.

\blem\label{lem:fibers} Let $F$ be a reduction of a singular complete fiber of a $\PP^1$-fibration of some smooth projective surface. Then $F$ is a rational snc-tree and its $(-1)$-curves are non-branching. Assume $F$ contains a unique $(-1)$-curve $L_F$. Then $F-L_F$ has at most two connected components and if it has two then one of them is a chain of rational curves. Moreover, $F$ contains exactly two components of multiplicity $1$, they are tips of $F$ and in case $F$ is not a chain they belong to the same connected component of $F-L_F$.
\elem

Note that sections meet only vertical components of multiplicity $1$. We need a description of snc-minimal boundary divisors of $\C^2$. We follow the proof by Daigle and Russell \cite[5.12]{Daigle_boundaries_of_C2}, \cite[\S1]{Russell_on_Ramanujam_thm} (which works for any surface completable by a chain). The lemma was originally proved by Ramanujam \cite{Ramanujam} using only the fact that $\C^2$ is a smooth contractible surface which is simply connected at infinity.

\blem\label{lem:D_for_C2_is_chain} If $(\ov X,D)$ is a smooth snc-minimal completion of $\C^2$ then $D$ is a chain. \elem

\begin{proof} First of all, consider a reduced divisor $B$ (on some smooth projective surface) which can be transformed into a $0$-divisor by a sequence of blowups and blowdowns by taking reduced total transforms and push forwards. Assume also that $D_0$ is either a zero divisor or a smooth curve not in $B$, such that the transformation can be done modulo $D_0$, i.e.\ that under all steps of the process the proper transform of $D_0$ is not contracted and stays smooth. We claim that $B$ contains a $(-1)$-curve in $B$ which is non-branching in $D_0+B$ (in particular, $D_0+B$ is not snc-minimal). To see this let $B_0$ be the first component of $B$ which is contracted by the transformation, i.e.\ the transformation starts with a sequence of blowups and then it contracts the proper transform $B_0'$ of $B_0$. If follows that $B_0'$ is a $(-1)$-curve which is non-branching in the total reduced transform of $D_0+B$ and hence that $B_0$ is a curve with $B_0^2\geq -1$, non-branching in $D_0+B$. But the intersection matrix of $B$, and hence of all its transforms, is negative definite, so $B_0^2=-1$ and we are done.

Suppose $D$ has a branching component $D_0$. By the factorization theorem for birational morphisms between smooth projective surfaces we know that $D$, being a boundary of $\C^2$, can be transformed into a chain of rational curves. Therefore, there is a connected component $B$ of $D-D_0$ which can be transformed to $0$ modulo $D_0$. By the above argument there is a $(-1)$-curve $B_0$ in $B$ which is non-branching in $D_0+B$. But then $D$ is not snc-minimal; a contradiction.
\end{proof}

The \emph{type} of an ordered chain of rational curves $T=T_1+\ldots+T_k$ is the sequence $[-T_1^2,\ldots,-T_k^2]$. We say that the chain is in a \emph{standard form} if it is of type $[0]$, $[1]$ or $[0,0,a_1,\ldots,a_{k-2}]$ for some $a_i\geq 2$. It is an elementary exercise to show that by blowing up and down on $T$ we can bring it into a standard form (the number of zeros is at most two by the Hodge index theorem). The formula \ref{eq:splitting_d} implies that if $T$ is a boundary of $\C^2$ in a standard form then it is of type $[0,0]$.

An snc-divisor is of \emph{quotient type} if it can be contracted algebraically to a quotient singularity, i.e.\ to a smooth or an isolated singular point which is locally analytically of type $\C^2/G$, where $G$ is a finite subgroup of $GL(2,\C)$. As a consequence, the intersection matrix of such a divisor is negative definite. Snc-minimal divisors of quotient type are well known, they are either negative definite chains of rational curves (corresponding to cyclic singularities $\C^2/\Z_k$) or special rational trees with unique branching components (forks). It is known that they do not contain $(-1)$-curves (see \cite{Brieskorn}). For a general snc-divisor $D$ we denote the set of its connected components of quotient type by $\qt(D)$.

Let $(\ov X,D)$ be a smooth pair. For it we can run a minimal model program to obtain a birational morphism onto a log terminal surface $(V,\Delta)$ such that there is no curve $L$ on $V$ for which $L^2<0$ and $L\cdot(K_{V}+\Delta)<0$. The pair $(V,\Delta)$ is called a \emph{minimal model} of $(\ov X,D)$. If $(\ov X,D)$ is snc-minimal (i.e.\ $D$ is an snc-minimal divisor) and the resulting morphism contracts only curves with supports in $D$ and its push-forwards then we say that $(\ov X,D)$ is \emph{almost minimal} (for another, more direct definition see \cite[\S 2.3.11]{Miyan-OpenSurf}). Recall that a \emph{log resolution} of a pair $(V,\Delta)$ with reduced $\Delta$ is a proper birational morphism from a smooth pair $r\:(\ov X,D)\to (V,\Delta)$ such that $D$ is the total reduced transform of $\Delta$.

Let $c(D)$ denote the number of connected components of $D$. The following proposition follows from \cite[6.20]{Fujita-noncomplete_surfaces}. We denote the logarithmic Kodaira-Iitaka dimension by $\kappa$.

\bprop\label{lem:log_exc_curves} Let $(X,D)$ be a smooth snc-minimal pair which is not almost minimal. Then there exists a $(-1)$-curve $\ell$ on $X$ which meets at most two connected components of $D$, each at most once and transversally and for which $\kappa(X\setminus D)=\kappa(X\setminus(D\cup\ell))$. In particular, if $p\:(X,D)\to (X',D')$, with $D'=p_*D$, is the contraction of $\ell$ then $$\chi(X'\setminus D')+c(D')=\chi(X\setminus D)+c(D)-1.$$ \eprop

Note that $(X',D')$ is a smooth pair and that if $(X',D')$ is not snc-minimal then the sum $\chi(X'\setminus D')+c(D')$ does not change when we snc-minimalize $D'$.

\section{Proof of Theorem 1}\label{sec:LZ}

Assume that $A\subseteq \C^2$ is a topologically contractible curve. Let $k\geq 0$ be the number of singular points of $A$. We write $\C^2$ as $\PP^2\setminus L_\8$, where $L_\8$ is the line at infinity. Let $\bar A\subseteq \PP^2$ be the closure $A$ and let $$\pi\:(\ov X',D')\to (\PP^2,L_\8+\bar A)$$ be the minimal log resolution of singularities. We denote the proper transforms of $\bar A$ and $L_\8$ on $\ov X'$ by $E$ and $L_\8'$ respectively. Since the germ of $\bar A$ at infinity is analytically irreducible, the reduced total transform of $L_\8$ contains a unique component $C'$ meeting $E$. Moreover, their difference can be written as $(\pi^*L_\8)_{red}-C'=D_1'+D_2$, where $D_1'$ and $D_2$ are connected and $D_1'$ contains $L_\8'$. We may, and shall, assume that $\bar A$ does not meet $L_\8$ transversally, otherwise $\bar A$ is a line in which case the above theorems obviously hold. By the minimality of the resolution it follows that $D_2$ is a rational chain with negative definite intersection matrix and with no $(-1)$-curves. In particular, $d(D_2)\geq 2$. Let $U$ be the reduced exceptional divisor over the singular points of $A$. Put $D_3=E+U$. We have $D'=D_1'+C'+D_2+D_3$.

\begin{figure}[h]\centering\includegraphics[scale=0.6]{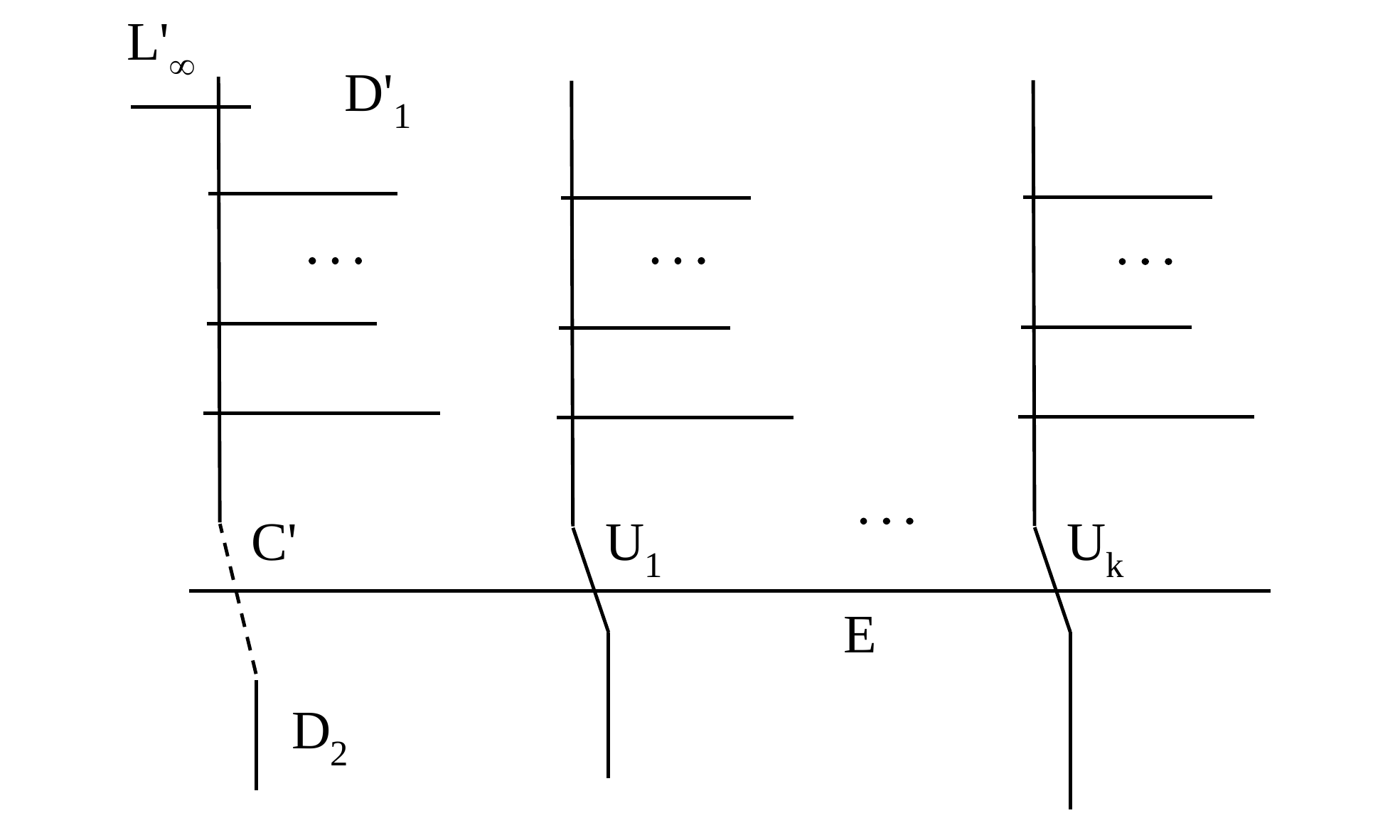}\caption{The divisor $D'$ on $\ov X'$. Lines denote chains of rational curves.}  \label{fig:Xbarprim}\end{figure}

It may happen that $L_\8'$ is a $(-1)$-curve (necessarily non-branching in $D_1'$). Moreover, its contraction may introduce new non-branching $(-1)$-curves in the boundary. Let $$\psi\:(\ov X',D')\to (\ov X,D)$$ be the composition of successive contractions of non-branching  $(-1)$-curves contained in $D_1'$ and its images. Put $D_1=\psi_*D_1$, and $C=\psi_*C'$. Since the curves contracted by $\psi$ are disjoint from $D_2+D_3$, we denote $D_2$, $E$, $U$, $D_3$ and their images on $\ov X$ by the same letters. We have $D=D_1+C+D_2+D_3$. Put $$B=D_1+D_2+D_3=D-C\text{\ \ and\ \ } X=\ov X\setminus B.$$ Clearly, $(\ov X,D_1+D_2+D_3)$ is a smooth completion of $X$ and the $D_i$'s are the connected components of the boundary (it may happen that $D_1=0$). Also, $X\setminus C=\C^2\setminus A$, with the smooth completion $(\ov X,D)$. It follows that $\chi(X)=\chi(\C^2\setminus A)+\chi(C\setminus (D_1\cup D_2\cup E))=2-\#C\cap(D_1\cup D_2\cup E)=-\#C\cap D_1.$ Thus, $\chi(X)=-1$, unless $D_1=0$.

Because the log resolution $\pi\:(\ov X',D')\to (\PP^2,L_\8+\bar A)$ is minimal, each connected component of $U$ contains a unique component $U_i$, $i=0,\ldots,k$, meeting $E$. Moreover, each $U_i$ is a $(-1)$-curve and $U_i\cdot E=1$. The divisor $B$ is snc-minimal except the case when $E^2=-1$ and $U$ has at most two connected components. Note however, that the minimality of the resolution implies that the only components of $B$ which meet $E$ are the $(-1)$-curves of $U$. So, even if $E$ is a non-branching $(-1)$-curve in $B$, its contraction does not introduce new non-branching $(-1)$-curves.

\begin{proof}[Proof of Theorem \ref{thm:main}]

\bcl If $D_1=0$ then Theorem \ref{thm:main} holds. \ecl

\begin{proof} We have $C^2\geq (C')^2+1\geq 0$. Let $D_C$ be the component of $D_2$ meeting $C$. Now \eqref{eq:splitting_d} gives $-C^2d(D_2)-d(D_2-D_C)=-1$, so $C^2d(D_2)+d(D_2-D_C)=1$. Because $d(D_2)\geq 2$, we obtain $C^2=0$ and $d(D_2-D_C)=1$, hence $D_2$ is irreducible. If we blow up once on $E\cap C$ and contract the proper transform of $C$ the new boundary of $X$ has the same dual graph but the self-intersection of $D_2$ increases. Repeating this elementary transformation we may assume $D_2^2=0$. Then the contraction of $U$ maps $\ov X$ to a smooth surface $\F$ with $\rho=2$ and the linear system of (the proper transform of) $D_2$ induces a projection $p\:\F\to \PP^1$. Moreover, (the proper transform of) $E$ is disjoint from $D_2$, so it is a $0$-curve. Then $U=0$ and Theorem \ref{thm:main} holds.
\end{proof}

Thus from now on we assume $D_1\neq 0$ (and $D_2\neq 0$).

\bcl There is no curve $\ell \not\subseteq D_1+D_2+U$ for which the intersection matrix of $\ell+D_1+D_2+U$ is negative definite. \ecl
\begin{proof} We have $\#(\ell+D_1+D_2+U)=\#D-1=\rho(\ov X)$, so the claim follows from the Hodge index theorem. \end{proof}

\bcl If $D_1$ is not negative definite then it is not a chain and $C$ is a $(-1)$-curve. \ecl

\begin{proof} Suppose $D_1$ is a chain. We change it into a standard form $\ti D_1$, so that the zero-curve is a tip of $D$. Denote the proper transform of $C$ by $\ti C$. If $\ti D_1$ is irreducible then it is a $0$-curve, so by \eqref{eq:splitting_d} $-1=d(\ti D_1+\ti C+D_2)=-d(D_2)\leq -2$, which is impossible. Thus $\ti D_1$ is not irreducible. Then $E+\ti C+D_2$ is vertical for the $\PP^1$-fibration induced by the $0$-tip, so either $\ti C$ is a branching component of a fiber or it meets two components of a fiber and the section contained in $\ti D_1$. By Lemma \ref{lem:fibers} $\ti C$ cannot be a $(-1)$-curve, hence $(\ti C)^2\leq -2$. Then $\ti D_1+\ti C+D_2$ is a boundary of $\C^2$ in a standard form, so it is of type $[0,0]$. But then $D_2=0$; a contradiction. Thus $D_1$ is not a chain. By Lemma \ref{lem:D_for_C2_is_chain} $D_1+C+D_2$ is not snc-minimal, so $C^2=-1$.
\end{proof}

\bcl $D_3$ is not a $(-1)$-curve. \ecl

\begin{proof} Suppose $D_3$ is a $(-1)$-curve. Then $U=0$ and $D_3=E$. Taking $\ell=E$ in Claim 2 we see that $D_1$ is not negative definite. By Claim 3 $D_1$ is not a chain and $C$ is a $(-1)$-curve. Consider the $\PP^1$-fibration given by the linear system of $E+C$. We have $h=2$, so $\Sigma_X=\nu$. The divisor $B$ contains no fibers. Indeed, otherwise $D$ would contain more than one fiber ($E+C$ is one of them), hence $D$ would contain a loop, which is false. We obtain $\Sigma_X=\nu=0$. Let $F$ be a singular fiber and $L_F$ its unique component not contained in $B$. The two sections contained in $B$ belong to different connected components of $B$, so the two components of $F$ of multiplicity one meeting them belong to different connected components of $F-L_F$. By Lemma \ref{lem:fibers} it follows that $F$ is a chain and meets the sections in tips. Since the vertical part of $D_2$ is connected, there is at most one singular fiber other than $E+C$. It follows that $D_1$ is a chain; a contradiction.
\end{proof}

\bcl If $U\neq 0$ then $D_3$ is not contained in a divisor of quotient type. \ecl
\begin{proof} Suppose $Q$ is a divisor of quotient type containing $D_3$. The $(-1)$-curve $U_1$ is branching in $D_3$, hence in $Q$. Because the self-intersection of $U_1$ is $(-1)$, the snc-minimalization of $Q$ does not touch $U_1$, hence leads to an snc-minimal divisor of quotient type which contains a branching $(-1)$-curve. But there are no such divisors. A contradiction.
\end{proof}

\ssk We now analyze the creation of the almost minimal model of $(\ov X,B)$. Let $\epsilon$ be the contraction of $E$ in case it is a non-branching $(-1)$-curve, otherwise put $\epsilon=\id_{\ov X}$. Let $$(\ov X,B)\xra{\epsilon}(\ov X_0,B_0)\xra{p_1}(\ov X_1,B_1)\xra{p_2}\ldots\xra{p_n}(\ov X_n,B_n)$$ be a sequence of birational morphism leading to the almost minimal model $(\ov X_n,B_n)$ of $(\ov X,B)$ grouped so that $p_{i+1}$ is a composition of a contraction of some $(-1)$-curve $\ell_i\not\subseteq B_i$ witnessing the non-almost minimalisty of $(\ov X_i,B_i)$ followed by the snc-minimalization of the image of $B_i$. Put $p=p_n\circ\ldots\circ p_1\circ \varepsilon$ and $X_i=\ov X_i\setminus B_i$. We denote $\ell_i$'s and their proper transforms on $\ov X$ by the same letters.

\bcl $\chi(X)=\chi(X_0)$ and $\chi(X_{i})\geq \chi(X_{i+1})$. \ecl
\begin{proof} Since $D_3$ is not a $(-1)$-curve, we have $\chi(X_0)=\chi(X)$. Suppose $\chi(X_{i+1})>\chi(X_i)$. Since $B_i$ is snc-minimal, $\ell_i$ meets two connected components of $B_i$, each transversally in a unique point, and together with these components contracts to a smooth point on $X_{i+1}$. Suppose $U\neq 0$. By Claim 5 these two connected components do not contain the image of $D_3$, so they contain the images of $D_1$ and $D_2$. Thus $D_1+\ell+D_2$ is contained in a divisor of quotient type disjoint from $D_3$, which contradicts Claim 1. Thus $U=0$ and Claim 1 implies that the connected components met by $\ell$ do not contain the image of $D_1$, hence contain images of $D_2$ and $E$. Let $\ov X\to \ti X$ be the contraction of $D_2$ and $E$.  We have $-1=d(D_1)\cdot d(C+D_2)-d(D_2)$, so since $d(D_2)\geq 2$, we see that $d(D_1)$ and $d(D_2)$ are coprime. In particular, $d(D_1)\neq 0$. Then the components of $D_1$ generate $H_2(\ti X,\Q)$. Using Nakai's criterion we show easily that $D$ supports an ample divisor (see (\cite[2.4]{Fujita-noncomplete_surfaces}), so $\ti X\setminus D_1$ is affine. But it contains the image of $\ell$, which is projective; a contradiction.
\end{proof}

\bcl $\kappa(X)=-\8$. \ecl

\begin{proof} Suppose $\kappa(X)\geq 0$. Then $\kappa(X_n)=\kappa(X)\geq 0$, so since $(\ov X_n,B_n)$ is almost minimal, the log BMY inequality (see \cite[3.4, \S 9]{Langer} and \cite[2.5]{Palka-exceptional}) gives $$0\leq \frac{1}{3}((K_{\ov X_n}+B_n)^+)^2\leq \chi(X_n)+\sum_{T\in \qt(B_n)} \frac{1}{|\Gamma(T)|}.$$ The divisor $B_n$ is snc-minimal, so each $\Gamma(T)$ is nontrivial, hence $0\leq \chi(X_n)+\frac{1}{2}\#\qt(B_n)\leq \chi(X)+\frac{1}{2}\#\qt(B_n)$. If $B_n$ has more than two connected components of quotient type then $D_1$, $D_2$ and $D_3$ are contained in disjoint divisors of quotient type, which is impossible by Claim $1$ (take $\ell=E$). It follows that all the above inequalities become equalities, so $B_n$ has exactly three connected components, two of them are of quotient type with $|G_i|=2$ and $\chi(X_n)=\chi(X)=-1$. It follows that two connected components of $B_n$ are $(-2)$-curves and $\chi(X_i)=\chi(X_{i+1})$ for every $i$. By Proposition \ref{lem:log_exc_curves} $n=0$, i.e.\ $(\ov X_0,B_0)$ is almost minimal.

If $U\neq 0$ then by Claim 5, $D_1$ and $D_2$ are of type $[2]$. But as we have seen in the proof of Claim 6, $d(D_1)$ and $d(D_2)$ are coprime. Thus $U=0$. By Claim 1 $E+D_1$ is not negative definite, so the only possibility is that $D_2$ and $E$ are $(-2)$-curves and $D_1$ is not negative definite. By Claim 3 $C^2=-1$. Consider the $\PP^1$-fibration of $\ov X$ induced by the linear system of $D_2+2C+E$. We have $h=1$ hence $0\leq\Sigma_X=\nu-1$. Since $D$ contains no loop, the $2$-section contained in $D_1$ meets $F_\8$ in one point. Since $D$ is snc-minimal we infer that $F_\8$ is of type $[2,1,2]$. Denote the middle $(-1)$-curve by $L$. When we snc-minimalize $D_1+C+D_2$ starting from the contraction of $C$ and $D_2$ we do not touch $F_\8-L$. By Lemma \ref{lem:D_for_C2_is_chain} the result of this minimalization is of type $[2,a,2]$ for some $a\leq 0$. However, the discriminant of $[2,a,2]$ is even, hence a chain of this type cannot be a boundary of $\C^2$; a contradiction.
\end{proof}

\bcl $X$ has a $\C^1$-fibration. \ecl

\begin{proof} Suppose $X$ has a $\PP^1$-fibration. It extends to a $\PP^1$-fibration of $\ov X$. Then $D_3$ is vertical, so it cannot contain a branching $(-1)$-curve. It follows that $U=0$. We have now $\Sigma_X=\nu-2$, so there are at least $2$ fibers contained in $B$. It follows that $D_1$ is a fiber, hence $d(D_1)=0$ and $d(D_1)$ and $d(D_2)$ are not coprime; a contradiction. Thus $X$ has no $\PP^1$-fibration. Suppose it also has no $\C^1$-fibration. Because $\kappa(X)=-\8$, the structure theorems for smooth surfaces of negative logarithmic Kodaira dimension imply that $(\ov X_n,B_n)$ is a minimal log-resolution of a log del Pezzo surface (\cite[2.3.15]{Miyan-OpenSurf}). Moreover, since not all connected components of $B$ are of quotient type, this log del Pezzo is open, hence has a structure of a Platonic $\C^*$-fibration by \cite{MiTs-PlatFibr}. In particular, $\chi(X_n)=0$. Then $\chi(X_n)>\chi (X)=-1$, which contradicts Claim~6.
\end{proof}

\bcl $X$ has a $\C^1$-fibration onto $\C^1$ with irreducible fibers. \ecl

\begin{proof}
Let $\ti \pi\:(\ti X,\ti B)\to (\ov X_0,B_0)$ be a minimal modification of $(\ov X_0,B_0)$ such that the above $\C^1$-fibration can be written as $r_{|X}$, where $r\:\ti X\to \PP^1$ is a $\PP^1$-fibration. Because the base point of $r\:\ov X_0\map \PP^1$ (if exists) belongs to $B_0$, we have $\rho(\ti X)-\#\ti B=\rho(\ov X_0)-\#B_0=0$, hence $h+\nu=2+\Sigma_X$ by \eqref{eq:Sigma}. But because $r_{|X}$ is a $\C^1$-fibration, $h=1$, so $\nu\geq 1$, i.e.\ $\ti B$ contains a fiber $F_\8$ of $r$. Suppose there is more than one fiber contained in $\ti B$. Since the reduced total transform of $D$, $\ti D=\ti B+C$, contains no loop, $C$ is vertical. In particular, $C^2\leq 0$. But $C$ is a branching component of $D$, so since $h=1$, it cannot be a fiber. Thus $C$ is a $(-1)$-curve, and hence it is a non-branching component of a fiber containing it. But $C$ is branching in $\ti D$, so it meets exactly two other vertical components of $\ti D$ and the section contained in $\ti D$. However, the former implies that its multiplicity in the fiber is at least two and the latter implies that its multiplicity is one; a contradiction. Thus $\nu=1$ and hence $\Sigma_X=0$, so $r(X)\cong \C^1$ and $r_{|X}$ has irreducible fibers.
\end{proof}

\bcl The $\C^1$-fibration of $X$ has no base points on $\ov X_0$.\ecl

\begin{proof} Denote the unique fiber and the unique section of $r$ contained in $\ti B$ respectively by $F_\8$ and $H$. The divisor $B_0$ is snc-minimal and, since $D_3\neq [1]$ and $D_1\neq 0$, it has three connected components. Let $T_1$, $T_2$, $T_3$ be the connected components of $\ti B$, say $T_3$ contains $H$. Then $T_3$ contains $F_\8$ and the divisors $T_1$ and $T_2$ are vertical and snc-minimal. After snc-minimalizing $T_3-H$ if necessary we may assume $\ti B-H$ contains only branching $(-1)$-curves. But then arguing as in the proof of the previous claim we see that in fact $T_3-H$ contains no $(-1)$-curves at all. Let $F$ be a singular fiber other than $F_\8$. Since $\Sigma_X=0$, we infer that $F$ contains a unique $(-1)$-curve $L_F$. By Lemma \ref{lem:fibers} $F-L_F$ has at most two connected components and one of them meets $H$. Since $T_1$, $T_2$ and $T_3$ are disjoint, there are at least two singular fibers other than $F_\8$. It follows that $H$ is a branching component of $\ti B$. Bu then $\ti \pi=\id$, i.e. the $\C^1$-fibration of $X$ is a restriction of a $\PP^1$-fibration of $\ov X_0$.
\end{proof}

Since $B_0$ is snc-minimal, $B_0-H$ contains no $(-1)$-curves, so $F_\8$ is a $0$-curve. It remains to prove that $F_\8=E$. Since $C$ is a is a branching component of $\epsilon_*D$ with $C^2\geq -1$, it follows that it is horizontal. Then $F_\8$ meets $C$. In particular, $F_\8$ is a tip of $B_0$ and $C$ is a section. If $F_\8\subseteq D_1$ then the snc-minimalization of $D_1+C+D_2$ does not contract $F_\8$, hence by Claim 3 leads to an snc-minimal boundary of $\C^2$ which is not a chain. But the latter is impossible by Lemma \ref{lem:D_for_C2_is_chain}. Since $D_2$ is negative definite, we get $F_\8\subseteq \epsilon_*D_3$. Since $U_i$ is branching in $D_3$, $\epsilon_*U_i$ is branching in $\epsilon_*U_i$, which implies that $F_\8$ is not one of the $\epsilon_*U_i$'s. Therefore, $\epsilon=\id$ and hence $F_\8=E$.

\end{proof}

We now show how Theorem A follows from Theorem \ref{thm:main}.

\begin{proof}[Proof of Theorem A] Let $(\ov X,D)$ and $r\:\ov X\to D$ be respectively a minimal smooth completion of $\C^2\setminus A$ and a $\PP^1$-fibration as in Theorem \ref{thm:main}. Denote the proper transform of $A$ on $\ov X$ by $E$. Since $(\ov X,D-E)$ is a smooth completion of $\C^2\setminus \Sing A$, $D-E$ has a unique connected component $D_\8$ which is a rational tree with non-negative definite intersection matrix, and such that $U=D-E-D_\8$ consists of $\#\Sing A$ connected components contractible to smooth points (of $\C^2$). In particular, connected components of $U$ are negative definite rational trees. By the minimality of $(\ov X,D)$ and by the analytical irreducibility of the singularities of $A$, each such tree contains a unique $(-1)$-curve $U_i$ and $E$ meets $U$ exactly in $U_i$'s, each once and transversally. Since the analytic branch of $A$ at infinity (considered, say, in $\PP^2$) is irreducible, there is a unique component $C$ of $D_\8$ meeting $E$. We obtain that $D$ is a rational tree with $\rho(\ov X)+1$ irreducible components. Note $D$ has $h=\#\Sing A+1$ horizontal components. We have $\Sigma_{\ov X\setminus D}=0$, so by \eqref{eq:Sigma} $h=3-\nu\leq 2$, so $A$ has at most one singular point. Clearly, $C$ is a horizontal component of $D$ and $D_\8-C$ has at most two connected components, call them $D_1$ and $D_2$. If $A$ is singular ($U\neq 0$) then $U_1$, the $(-1)$-curve of $U$ meeting $E$, is the second horizontal component of $D$. They are both sections of $r$. Because $(\ov X,D)$ is minimal, $D-C-U_1$ contains no $(-1)$-curves. Indeed, such a curve would be a non-branching component of a fiber and, since the horizontal components of $D$ are sections, also a non-branching $(-1)$-curve in $D$.

Up to this point we just reproved for $(\ov X,D)$ what could be obtained by taking the special minimal completion of $\C^2\setminus A$ as defined in section $1$. Suppose $U=0$. Then $h=1$, so $\nu=2$, i.e.\ there is a unique fiber of $r$ contained in $D_1+D_2$. Since $D-C$ contains no $(-1)$-curves, the fiber is a $0$-curve. It follows that, say, $D_1$ is a $0$-curve. Making an elementary transformation on $D_1$ we may assume $C^2=-1$. The snc-minimalization of $D_\8$ does not contract $D_1$, which by Lemma \ref{lem:D_for_C2_is_chain} implies that $D_2$ is a chain (negative definite or empty). Since $D_\8$ is a boundary of $\C^2$, \eqref{eq:splitting_d} gives $-1=d(D_1)\cdot d(C+D_2)-d(D_2)=-d(D_2)$, so $d(D_2)=1$. Thus $D_2=0$ and hence $\ov X$ is a Hirzebruch surface. The contraction of $C$ maps it to $\PP^2$ and $C+D_1$ into a pair of lines, so we are done.

We may therefore assume that $U\neq 0$. Then $A$ has a unique singular point. Let $F$ be a singular fiber of $r$. Its unique component $L_F$ not contained in $D$ is also the unique $(-1)$-curve in $F$. Now $C$ and $U_1$ are sections of $r$, so they meet components of $F$ of multiplicity one. By Lemma \ref{lem:fibers} $F-L_F$ has at most two connected components. It follows that $F$ is a chain. Indeed, otherwise only one connected component of $F-L_F$ contains components of multiplicity $1$, which would imply that $D$ contains a loop. Thus, every singular fiber of $r$ is a chain with a unique $(-1)$-curve. Also, $F-L_F$ has exactly two connected components, both contained in $D$. Each such chain contains exactly two components of multiplicity one, which are tips of the chain. Since $U$ can be contracted to a point by iterating contractions of $(-1)$-curves, it follows that $U_1$ meets exactly two components of $U-U_1$, and hence $U_1\cdot (D-U_1)=3$. Thus $r$ has exactly two singular fibers, $F_1$ and $F_2$. Let $L_i$ be the unique $(-1)$-curve of $F_i$. We have $D-E=D_\8+U$ and we can write $U-U_1=V_1+V_2$ and $D_\8-C=D_1+D_2$, so that $D_i$ and $V_i$ are connected and $F_i=V_i+L_i+D_i$. Put $Y=\ov X\setminus D_\8$. The morphism $\pi_{|Y}\:(Y,U+E\setminus \{\8\})\to (\C^2,A)$ is a log resolution of singularities. The curves $\pi(L_1\cap Y)$ and $\pi(L_2\cap Y)$ are isomorphic to $\C^1$ and meet in one point, transversally.

We claim there exist coordinates $\{x_1,x_2\}$ on $\C^2$ such that $\pi(L_i\cap Y)$ is given by $x_i=0$. To see this first contract $U$. The images of $L_1$ and $L_2$ are smooth, meet transversally and have non-negative self-intersections. Now blow up on $(L_1+L_2)\cap D_\8$, so that the proper transforms of $L_1$ and $L_2$ are again $(-1)$-curves and denote the resulting projective surface by $\ti X$, the total reduced transform of $D_\8$ by $\ti D_\8$ and the proper transforms of $L_i$ by $\ti L_i$. By construction, $\ti D_\8$ is a chain met by $\ti L_i$ in a tip $W_i$ which is a $(-1)$-curve. Also, $(\ti X,\ti D_\8)$ is a smooth completion of $\C^2$. For the $\PP^1$-fibration $\ti X\to \PP^1$ induced by the linear system of $\ti L_1+\ti L_2$ we have $\Sigma_{\C^2}\geq 1$ and $h=2$, so by \eqref{eq:Sigma} $\nu=\Sigma_{\C^2}\geq 1$, i.e.\ $\ti D_\8$ contains a fiber $F$. Because $D_\8-W_1-W_2$ is connected and does not meet $\ti L_1+\ti L_2$, it follows that $F=D_\8-W_1-W_2$. Let $W_3$ and $W_4$ be the (different) components of $F$ meeting $W_1$ and $W_2$ respectively. Contracting successively $(-1)$-curves in $F-W_3-W_4$ if necessary we may assume there are no $(-1)$-curves in $F-W_3-W_4$. Because $W_3$ and $W_4$ meet sections, they have multiplicity $1$, so then $F$ is necessarily of type $[1,2,\ldots,2,1]$, where the subchain of $(-2)$-curves has length $s\geq 0$. Since $W_1+F+W_2$ is a boundary of $\C^2$, its discriminant is $-1$, which gives $s=0$. Blowing up once on $W_1+W_2$ we may assume $F$ is of type $[1,2,1]$. Then the contraction of $W_1+W_2+W_3+W_4$ maps the completion of $\C^2$ onto $\PP^2$ and $L_1+L_2$ onto a pair of lines meeting transversally. This gives the coordinates $\{x_1,x_2\}$.

Put $n=d(V_1)$ and $m=d(V_2)$. The morphism $r_{|Y}\:Y\to \PP^1$ is a $\C^1$-fibration and $x_1^n/x_2^m$ is a coordinate on $\C^1$. Since $E\setminus \{\8\}$ is a fiber of $r_{|Y}$, it has equation $x_1^n/x_2^m=\alpha$ for some $\alpha\in \C^*$ and we may assume $\alpha=1$. Then $A$ has equation $x_1^n=x_2^m$.
\end{proof}

\section{Another proof of the Abhyankar-Moh-Suzuki theorem}\label{sec:AMS}

We now give another, independent of section \ref{sec:LZ}, proof of Theorem B. We need the following lemma. In case of contractible surfaces it was obtained by similar methods by Gurjar and Miyanishi \cite{GM_preprint}.

\begin{lemma}\cite[3.1(vii)]{Palka-classification1}\label{lem:Q-acyclic}  Let $X\to X'$ be a log resolution of a rational $\Q$-acyclic normal surface, let $\widehat E$ be the reduced exceptional divisor and $(\ov X,D)$ a smooth completion of $X$. If $\widehat E+D$ is a sum of rational trees then $$|d(D)|=d(\widehat E)\cdot |H_1(X',\Z)|^2.$$ \end{lemma}

\begin{proof} Let $M_D$ and $M$ be the boundaries of closures of tubular neighborhoods of $D$ and $\widehat E$. We may assume that $M_D$ and $M$ are disjoint oriented 3-manifolds. Since $\widehat E$ is a sum of rational trees, $H_1(\widehat E,\Q)=0$. By the $\Q$-acyclicity of $X'$ the components of $D+\widehat E$ freely generate $H_2(\ov X,\Q)$, so $d(D+\widehat E)\neq 0$. By \cite{Mumford} $H_1(M_D,\Z)$ and $H_1(M,\Z)$ are finite groups of orders respectively $|d(D)|$ and $d(\widehat E)$. By the Poincare duality $H_2(M_D,\Z)$ and $H_2(M,\Z)$ are trivial. Let $K=\ov X\setminus (Tub(D)\cup Tub(\E))$. By the Lefschetz duality $H_i(K,M_D)\cong H^{4-i}(K,M)=H^{4-i}(X',\Sing\ X'),$ which for $i>1$ implies that $H_i(K,M_D)\cong H^{4-i}(X')\cong H_{3-i}(X')$ by the universal coefficient formula. Thus the reduced homology exact sequence of the pair $(K,M_D)$ with $\Z$-coefficients gives: $$0\raa H_2(K)\raa H_1(X')\raa H_1(M_D)\raa H_1(K)\raa H_2(X')\raa 0.$$  On the other hand, since $H_i(K,M)\cong H_i(X',\Sing\ X')$ and $H_1(X',\Sing\ X')=H_1(X')\oplus \wt H_0(\Sing X')$, the reduced homology exact sequence of the pair $(K,M)$ gives: $$0\raa H_2(K)\raa H_2(X')\raa H_1(M)\raa H_1(K)\raa H_1(X')\raa 0.$$
From the two exact sequences we obtain $|H_2(K)|\cdot |d(D)|\cdot |H_2(X')|=|H_1(X')|\cdot |H_1(K)|$ and $|H_2(K)|\cdot d(\E)\cdot |H_1(X')|=|H_2(X')|\cdot |H_1(K)|$, hence $$|d(D)|\cdot |H_2(X',\Z)|^2=d(\E)\cdot |H_1(X',\Z)|^2.$$
Because a rational $\Q$-acyclic surface is necessarily affine by an argument of Fujita (\cite[2.4]{Fujita-noncomplete_surfaces}), $X$ is an affine variety, and hence has a structure of a CW-complex of real dimension $2$. It follows that $H_2(X',\Z)$ is torsionless, hence $H_2(X',\Z)=0$.

\end{proof}

Assume $A\subseteq \C^2$ is a smooth contractible planar curve.  Let $(\ov X,D)$ be a completion of $\C^2\setminus A$ defined in section \ref{sec:LZ}. We have $U=0$, so $D_3=E$. Assume that $D_1\neq 0$. Let $D_C$ be the component of $D_1$ meeting $C$. Because $D_1+C+D_2$ is a boundary of $\C^2$, its discriminant is $-1$.  By \eqref{eq:splitting_d} $$-1=d(D_1+C+D_2)=d(D_1)\cdot d(C+D_2)-d(D_1-D_C)\cdot d(D_2),$$ which gives $\gcd(d(D_1),d(D_2))=1$. Since $d(D_2)\geq 2$, we infer that $d(D_1)\neq 0$.

Suppose $E^2>0$. By blowing up on $E\setminus C$ we may replace it with a chain $\ti E+H+D_4$, where $\ti E$, the proper transform of $E$, is a $0$-curve, $H$ is a $(-1)$-curve and $D_4$ is a chain of $(-2)$-curves of length $E^2-1$. Now $\ti E$ induces a $\PP^1$-fibration of the constructed projective surface, such that $H$ is the unique horizontal component of $\ti B=D_1+D_2+\ti E+H+D_4$. Since $H$ and $D_1+D_2$ are disjoint, $\ti E$ is the unique fiber contained completely in $\ti B$. Then \eqref{eq:Sigma} gives $\Sigma_X=0$. It follows that every singular fiber $F$ contains a unique component $L_F$ not contained in $\ti B$ and this component is a unique $(-1)$-curve of $F$. Since $H$ is a section of the fibration, by Lemma \ref{lem:fibers} $L_F\cdot H = 0$ and $F-L_F$ has at most two connected components. One of these components (the one meeting $H$, necessarily non-empty) is contained in $D_4$. But $D_4$ is connected, so we see that there is at most one singular fiber. Then $D_1+D_2$ is contained in this fiber, hence it is connected, so $D_1=0$, in contradiction to the assumption.

Suppose $E^2<0$. Then $d(B)=d(D_1)d(D_2)d(E)\neq 0$, so the components of $B$ are independent in $H_2(\ov X,\Q)$, hence they generate freely the latter space. Let $X\to X'$ be the contraction of $E$ and $D_2$. We check using Lefschetz duality and standard exact sequences that $X'$ is $\Q$-acyclic. Applying Lemma \ref{lem:Q-acyclic} to $(\ov X,D_1)$ and $\widehat E=D_2+E$ we get that $d(D_2+E)$ divides $d(D_1)$. But $d(D_2+E)=d(D_2)\cdot d(E)$, so $d(D_2)$ divides $d(D_1)$; a contradiction.

Thus $E$ is a $0$-curve. Then $D_1$ and $D_2$ are vertical for the $\PP^1$-fibration of $\ov X$ induced by the linear system of $E$. Since $D_2$ is negative definite, we have $\nu\leq 2$. By \eqref{eq:Sigma} $\nu=2+\Sigma_X$, so $\nu=2$. This means that $D_1$, being snc-minimal, is a $0$-curve, so $d(D_1)=0$; a contradiction.

Therefore, we proved that $D_1=0$. As in the Claim 1 in the previous section we argue that $C^2=0$ and $D_2$ is irreducible. Then $\rho(\ov X)=2$, so $\ov X$ is a Hirzebruch surface. Again, after making some elementary transformation we may assume that $D_2^2=-1$. Then the contraction of $D_2$ maps $\ov X$ onto $\PP^2$ and $E$ onto a line. Theorem B follows.

\bibliographystyle{amsalpha}
\bibliography{bibl}
\end{document}